\documentclass[11pt, reqno]{amsart}
\usepackage{amssymb,amsmath,amsthm,amsfonts}

\usepackage[pagewise]{lineno}

\usepackage{tikz-cd}
\usepackage{eucal}

\usepackage{graphicx}

\usepackage{caption}
\usepackage{blindtext}
\usepackage{enumitem}

\addtocontents{toc}{\setcounter{tocdepth}{-1}} 

\newenvironment{psmallmatrix}
  {\left(\begin{smallmatrix}}
  {\end{smallmatrix}\right)}

\newenvironment{bsmallmatrix}
  {\left[\begin{smallmatrix}}
  {\end{smallmatrix}\right]}

\newcommand{\widesim}[2][1.5]{
  \mathrel{\underset{\scriptscriptstyle #2}{\scalebox{#1}[1]{$\sim$}}}
}

\newtheorem{thm}[subsection]{Theorem}
\newtheorem{defn}[subsection]{Definition}
\newtheorem{lemma}[subsection]{Lemma}
\newtheorem{cor}[subsection]{Corollary}

\theoremstyle{definition}

\newtheorem{rmk}[subsection]{Remark}

\DeclareMathOperator{\Cl}{Cl}
\DeclareMathOperator{\Epin}{Epin}
\DeclareMathOperator{\Spin}{Spin}
\DeclareMathOperator{\rank}{rank}
\DeclareMathOperator{\End(O)}{End(O)}

\date{\vspace{-5ex}}

\begin{document}
\title {Two Lives : Compositions of Unimodular rows}

\author{Vineeth Chintala}
\address{University of Hyderabad, Hyderabad-500046, Telangana, India.}

\begin{abstract}
The paper lays the foundation for the study of unimodular rows using Spin groups. We show that elementary orbits of unimodular rows (of any length $n\geq 3$) are equivalent to elementary Spin orbits on the unit sphere. (This bijection is proved over all commutative rings). In the special case $n=3$, this leads to an interpretation of the Vaserstein symbol using Spin groups.

  In addition, we introduce a new composition law that operates on certain subspaces of the underlying quadratic space (using the multiplication in composition algebras). In particular, the special case of split-quaternions leads to the composition of unimodular rows (discovered by L. Vaserstein and later generalized by W. van der Kallen). Strikingly, with this approach, we now see the possibility of new orbit structures not only for unimodular rows (using octonion multiplication) but also for more general quadratic spaces.
\end{abstract}

\date{}


\maketitle

\section{Introduction}

When multiple research areas evolve around the same object, one expects that there is a connection between them. The more distinct the methods are, the more fruitful the connection will be. 
In this paper, we will explore this double life for \emph{unimodular rows}. Though unimodular rows are primarily used as a tool to study Projective modules, we will see here that they can also be fruitfully employed from the perspective of Quadratic forms and Spin groups. One consequence of this approach is that it gives a neat interpretation of some surprising results like the Vaserstein symbol, through a simple composition law operating in the background. On the other hand, we arrive at new questions in this development via quadratic forms. We currently know (through the work of W. van der Kallen~\cite{vdk2} and others) that the Vaserstein symbol can be generalized to a group law on certain (higher-dimensional) orbit-spaces of unimodular rows. But now, when interpreted as a result in quadratic forms, there is the exciting possibility that such group laws may generalize beyond hyperbolic quadratic spaces. In particular, we see (in Part C) that Vaserstein composition corresponds to the special case of split quaternions, and if we pick another composition algebra, we get a different composition rule. 

Let $R$ be any commutative ring. A vector $v = (a_1, \cdots, a_n) \in R^n$ is called a unimodular row (of length $n$) if $v \cdot w^\intercal  =1$ for some $w= (b_1,\cdots, b_n) \in R^n.$ Let $Um_n(R)$ denote the set of unimodular rows of length $n.$ Here are a few places where unimodular rows turn up.

\vskip 5mm \subsection{First life : Cancellation of Projective modules}
 The study of projective modules is one of the primary motivations to investigate unimodular rows. Given a vector $v \in R^n$, consider the map $P_v : R^n \rightarrow R$, given by 
\[P_v(w) = v\cdot w^\intercal .\]

Whenever $v $ is a unimodular row, the kernel of $P_v$ becomes a stably free module of rank $n-1$ (and hence a finitely generated projective module).  Note that the group $GL_n(R)$ (and hence any subgroup) acts on the right on the set of unimodular rows $Um_n(R).$
One way to show that the projective module $ker(P_v)$ is free is to show that $v$ appears as a row in an invertible matrix.   Thus the interest in unimodular rows began, and grew with the Quillen-Suslin theorem (also known as Serre's problem) which states that finitely generated projective modules over polynomial-rings are free - Quillen received a Fields medal in 1978 in part for his proof of the theorem (see \cite{Lam}).  
As one goes beyond polynomial rings, the orbits may not be trivial, leading naturally to the study of quotients such as $Um_n(R)/SL_n(R)$ and $Um_n(R)/E_n(R)$ and there is a rich array of results stating conditions under which these orbit spaces have an abelian group structure (see for example \cite{DTZ, SV, vdk2}). As we will see, these same orbits can also be examined from a different point of view, as Spin-orbits on the unit sphere.

\vskip 5mm \subsection{Second life : Group structures on spheres.} 

  Consider the space $H(R^n) = R^n \oplus (R^n)^*$, equipped with a quadratic form $$q(x,y) = x\cdot y^\intercal .$$  
  
Let $U_{2n-1}(R)$ denote the unit sphere, i. e., the set of all elements $(v,w) \in H(R^n)$ such that $v \cdot w^{\intercal}=1.$
Suppose there is another element $w' \in R^n$ such that $v\cdot w'^\intercal  = v\cdot w^\intercal =1.$ Then it turns out (Theorem \ref{thm1}) that the two points on the unit sphere - $(v,w)$ and $(v,w')$ - lie on the same orbit under the action of $\Epin_{2n}(R)$, a normal subgroup of $\Spin_{2n}(R)$ called the elementary Spin group. (We will describe the elementary Spin group in Section 3).  

\vskip 2mm
\noindent This gives us the map
\[ v \rightarrow \frac{(v,w)}{\Epin_{2n}(R)}. \]

\noindent We will show that the kernel of the above map is the orbit of $v$ under the action of the elementary linear group $E_n(R).$
\[ \frac{v}{E_n(R)} \longleftrightarrow \frac{(v,w)}{\Epin_{2n}(R)} \]
The main result of the paper (Theorem \ref{main}) is the following bijection
\[ \frac{Um_n(R)}{E_n(R)} \longleftrightarrow \frac{U_{2n-1}(R)}{\Epin_{2n}(R)} =  \frac{U_{2n-1}(R)}{EO_{2n}(R)} \]
In short, orbits of unimodular rows can be studied as orbits of points on the unit sphere whose geometry is more familiar to us. 
The seminal paper of Suslin-Vaserstein \cite{SV}  contains some hints of the above bijection for $n=3$ (though they don't talk about Spin groups). In this paper, we will prove this bijection and generalize it to any $n \geq 3$ (Theorem \ref{main}). 

\vskip 1mm
The same paper of Suslin-Vaserstein \cite[Section 5]{SV} also introduced what is now referred to as the Vaserstein symbol, which revealed a surprising symplectic structure for the orbit spaces $\frac{Um_3(R)}{E_3(R)}.$ When $n=3,$ we have $\Epin_6(R) \cong E_4(R)$ which will be used to show that there is a bijection between $\frac{Um_3(R)}{E_3(R)}$ and the $E_4(R)$-orbits of $4\times 4$ alternating matrices with Pfaffian $1$ (where $ M \sim gMg^T$ for $g \in E_4(R)$). This gives an interpretation of  the Vaserstein symbol using Spin groups (see Part B and Theorem \ref{vaserstein}).

\vskip 2mm

The second contribution of this paper is the introduction of  a new composition law that holds on certain subspaces of the hyperbolic space $H(R^n) =R^n \oplus (R^n)^*$ (using quaternion multiplication). This composition law (on matrices) generalizes the Vaserstein composition for unimodular rows (see Part B and Remark \ref{rmk1} in Part C).

One quality of the composition law is that it is independent of the dimension of $R$, whereas it seems necessary to place such restrictions on the base ring $R$ to get group structures on orbit spaces of unimodular rows (in some situations). What is the reason for this dependence on the dimension of $R$, and how do we arrive at \emph{this} dependence : in this case, $\dim(R) \leq 2n- 3$? Looking back, there are two results that hint at a general composition law operating in the background - one is the Vaserstein symbol (see Part B), the other being the Mennicke-Newman Lemma (see \cite[Lemma 3.2]{vdk3}) that essentially says that, under the above dimension restrictions, one can project two points of the unit sphere $U_{2n-1}(R)$ onto the same $(n+1)$-dimensional subspace where the composition law operates (see Remarks \ref{rmk1}, \ref{rmk2}). 

This investigation using quadratic forms opens the door for research in two striking general directions : 
\begin{enumerate}[label={\alph*.}]
\item  In Section \ref{octonioncl}, we use the multiplication of split-octonions to define a (nonassociative) composition law on $(n+4)$-dimensional subspaces of $H(R^{n})$, suggesting that there may be a quasigroup structure on orbits of unimodular rows. 

\item Let $(V,q)$ be any quadratic space, and $U$ the set of unit vectors of $V$ ($q(x) = 1$). When is there a group structure on the orbit spaces $\frac{U}{Spin(V)}$?
\end{enumerate}

\vskip 5mm \subsection{The other lives of unimodular rows} 
The vector $v =(a,b,c)$ also corresponds to coefficients of the quadratic form $ax^2 + bxy +cy^2.$ The condition $v\cdot w^\intercal  = 1$ can then be seen as a restriction to \emph{primitive} quadratic forms. In the study of unimodular rows, one is mainly concerned with $SL_3(R)$ orbits of $Um_3(R)$, whereas Gauss's composition gives a group structure on the $SL_2(\mathbb{Z})$ orbits of binary quadratic forms. It is known that  Gauss's composition extends to an arbitrary base ring (see \cite{K} and \cite{W}).  It is also known that if $\frac{1}{2} \in R$ and the discriminant $ b^2 -4ac$ is a square, then the unimodular row $(a,b,c)$ is completable and the corresponding projective module is free (see \cite{KM} or \cite{Ko}). But it is not known whether there are deeper connections between projective modules and M. Bhargava's work \cite{Bh1} on the composition laws for quadratic and higher forms - an intriguing line to pursue, that hopefully future research can shed some light on.  
 
\begin{rmk}
There are many other active areas related to unimodular rows -  notably, Euler class groups (\cite{BRS, DTZ}), Grothendieck-Witt groups (\cite{FRS}), $\mathbb{A}^1$-homotopy theory  (\cite{AF1, AF2, AF3}) and Suslin Matrices (see \cite{RJ} for a survey). 

 Under certain smoothness conditions for a ring $R$ with (Krull) dimension $d$,  J. Fasel has given an interpretation of $\frac{Um_{d+1}(R)}{E_{d+1}(R)}$ in terms of cohomology (see \cite{F1}). More recently this quotient space has been explicitly computed in \cite{DTZ} for some rings. 
Ravi Rao and Selby Jose have written a series of papers (\cite{JR1, JR2}) examining general quotients  $\frac{Um_n(R)}{E_n(R)}$ by studying the algebraic properties of Suslin matrices. As you can tell, the behaviour of the quotient $\frac{Um_n(R)}{E_n(R)}$ depends on the base ring $R$ (especially its Krull dimension), and there is a continuing trend simplifying the hypothesis on the base ring to construct and analyze the structure of the orbit spaces (see for example \cite{FRS, GRK, GGR, Gu, SS} or Part II of the recent conference proceedings \cite{AHS}). 
\vskip 2mm

The Vaserstein symbol gives a symplectic structure to the orbit-spaces of unimodular rows and plays an important role in the study of stably-free modules. It was first introduced in \cite[Section 5]{SV} where orbits of unimodular rows were investigated under the action of both linear and symplectic groups. Further investigation of the symplectic orbits can be found in \cite{CR1, CR2, TS2}. The recent work of T. Syed \cite{TS1} generalizes the Vaserstein symbol to study the orbit spaces $\frac{Um(R+ P)}{E(R+P)}$, where $P$ is a rank-$2$ projective module with a fixed trivialization of its determinant. 

A. Asok and J. Fasel have provided an interpretation of the Vaserstein symbol in terms of $\mathbb{A}^1$-homotopy theory (see \cite{F2}) and we will explain this connection briefly in Part B.  In \cite[Theorem 7.5]{FRS} the Vaserstein symbol was used to prove that stably free modules of rank $d-1$ are free under certain smoothness conditions ($R$ is a smooth affine $k$-algebra of dimension $d\geq 3$, where $k$ is an algebraically closed field and $\frac{1}{(d-1)!} \in k$), thus settling a long-standing question of A. Suslin.   
 
Perhaps some day,  another mathematician will write a ``many lives" generalization of this paper. 
\end{rmk}

\tableofcontents
 \subsection{Overview}
The paper is broken down into three parts and can be read non-linearly. A reader whose main interest is unimodular rows may begin with Part A, where the connection to Spin groups is explored in detail. Alternatively, a person who is curious about general quadratic forms may find it profitable to look at Part C first, where a new composition law is defined using the multiplication in composition algebras. Here the Vaserstein composition (for unimodular rows) corresponds to the special case of split quaternions. Finally, those who are comfortable with both the worlds and prefer to quickly know what is going on, may begin with Part B which acts as a bridge (examining the Vaserstein symbol), and then read around accordingly.

Essentially the paper makes two contributions. First, we look at (elementary) orbits of unimodular rows and prove that they correspond to (elementary) Spin-orbits on the unit sphere.  Secondly, we introduce a composition law - that holds in certain $(n+2)$-dimensional subspaces of $H(R^n).$ This composition (in terms of matrices) follows a simple recursive rule, starting with the multiplication of split-quaternions. When $n=3$, it has the \emph{same} properties as the composition law (on unimodular rows) discovered by L. Vaserstein. For general $n$, it describes in a simple matrix form, the composition introduced by van der Kallen using what are known as weak Mennicke symbols. This lays the foundation for the study of unimodular rows using Spin groups. The general formulation of the composition law also raises the possibility of new orbit structures using octonion multiplication.

\subsection{Notation}
All modules $V$ defined in the paper are free $R$-modules over some commutative ring $R.$ No condition is placed on $R$ so the results in the paper hold for all commutative rings. A quadratic form $q$ on $V$ is said to be \emph{non-degenerate} when the corresponding bilinear form is non-degenerate, i. e., $\langle x,v\rangle =0 $ holds for all $v \in V$, 
if and only if $x=0$. (Here $\langle v,w\rangle = q(v+w) - q(v) - q(w)$ for $v,w \in V.$)

%

{\centering \Large \part{The bijection between (elementary) Spin-orbits on the sphere and  the elementary orbits of unimodular rows}}
\vskip 5mm

\section {Preliminaries : From Clifford algebra to Suslin matrices}
\vskip 2mm

For general literature on Clifford algebras and Spin groups over a commutative ring, see \cite{B2,Kn}.
\vskip 1mm

\subsection{Clifford Algebras}
Let $V$ be a free $R$-module where $R$ is any commutative ring. If we equip $V$ with a non-degenerate quadratic form $q$, then $(V,q)$ is called a quadratic space. The algebra $\Cl(V,q)$ is the ``freest" algebra generated by $V $ subject to the condition 
$x^2 = q(x)$ for all $x \in V.$ More precisely, $\Cl(V,q)$ is the quotient of the tensor algebra $$T(V) = R \oplus V\oplus V^{\otimes 2} \oplus \cdots \oplus V^{\otimes n} \oplus \cdots$$ 
by the two sided ideal $I(V,q)$ generated by all the elements $x\otimes x - q(x)$ with $x \in V.$ 
\vskip 2mm
\noindent For the purpose of this article, we only need to know the following basic properties of Clifford algebras :
\begin{itemize}
\item \emph{$\mathbb{Z}_2$-grading}:  Grading $T(V)$ by even and odd degrees, it follows that the Clifford algebra has a $\mathbb{Z}_2$-grading $\Cl(V,q) =  \Cl_0\oplus \Cl_1$ such that $V \subseteq \Cl_1$ and $\Cl_i\Cl_j\subseteq \Cl_{i+j}$ ($i,j$ mod $2$). 

\vskip 1mm
\item  \emph{Universal property}: Let $ i : V \hookrightarrow Cl(V,q)$ denote the inclusion map and $A$ be any associative algebra over $R.$ Then any linear map 
$j : V \rightarrow A$ such that 
\[ j(x)^2 = q(x) \text{ for all $x \in V$},\]
lifts to a unique $R$-algebra homomorphism $f : \Cl(V,q) \rightarrow A$ such that 
$f \circ i = j.$

\item \emph{Basis of $\Cl$} : The elements of $V$ generate the Clifford algebra. Furthermore, the following result implies that if $\rank(V) =n$, then $\rank(Cl) = 2^n.$

 \begin{thm}(Poincar\'e-Birkhoff-Witt)
 \\Let $\{v_1,\cdots, v_n\}$ be a basis of $(V,q).$ Then $\{v_1^{e_1}\cdots v_n^{e_n} : e_i =0,1\}$ is a basis of $\Cl(V,q).$ 
  \end{thm}
\noindent For a simple proof, see \cite[Theorem IV. 1.5.1]{Kn}.
\end{itemize}

\vskip 3mm

Let $Cl$ denote the Clifford algebra of the quadratic space $H(R^n) := R^n \oplus R^{n*}$, with $q(v,w) = v\cdot w^\intercal  .$  We will now give an explicit representation of $Cl \cong M_{2^n}(R)$ using what are called Suslin matrices. For a detailed exposition, see \cite{CV1}.

\vskip 5mm
 \subsection{Suslin matrices}\label{suslin}

For any two vectors $ v = (a_1, \cdots, a_n)$ and $w = (b_1, \cdots, b_n)$ in $R^n$, the Suslin matrix $\mathcal{S}_{n-1}(v,w)$ is defined as follows : 

For $n=2$, define
\[  \mathcal{S}_1(v,w) =  \begin{psmallmatrix}
a_1     &a_2    \\
-b_2    & b_1 \end{psmallmatrix} \hskip 5mm   \overline{\mathcal{S}_1(v,w)} =  \begin{psmallmatrix}
b_1    & -a_2    \\
b_2    & a_1 \end{psmallmatrix}   \]

\noindent For the general case, write $v = (a_1, v')$ and $w = (b_1, w')$ with $v', w' \in R^{n-1}.$ Then
 \[ \mathcal{S}_{n-1}(v,w) = \begin{bsmallmatrix}
a_1      &  \mathcal{S}_{n-2}(v',w') \\\\
-\overline{\mathcal{S}_{n-2}(v',w')}      &  b_1 \end{bsmallmatrix},
\hskip 5mm 
\overline{\mathcal{S}_{n-1}(v,w)} =  \begin{bsmallmatrix}
b_1      &  - \mathcal{S}_{n-2}(v',w') \\\\
\overline{\mathcal{S}_{n-2}(v',w')}      &  a_1\end{bsmallmatrix}\]
\vskip 2mm

The matrix $\mathcal{S}= \mathcal{S}_{n-1}(v,w)$ has size $2^{n-1} \times 2^{n-1}$
 and has the following properties :
 \begin{enumerate}[label={\alph*.}]
 \item $\overline{\mathcal{S}(v,w)} = \mathcal{S}(w,v)^\intercal .$
 \item $\mathcal{S}\overline{\mathcal{S}} = \overline{\mathcal{S}}\mathcal{S}= (v\cdot w^\intercal ) I_{2^{n-1}}.$
 \end{enumerate}
\vskip 2mm

\noindent In his paper \cite{S}, A. Suslin then describes a sequence of matrices $J_n \in M_{2^n}(R)$ by the recurrence formula

 \begin{equation*}
J_n = \begin{cases}
1      &\text{for $n = 0$}\\
\\
\begin{psmallmatrix}
J_{n-1}         & 0\\
0         & -J_{n-1}\end{psmallmatrix} &\text{for $n$ even}\\
\\
 \begin{psmallmatrix}
0         & J_{n-1}\\
-J_{n-1}        & 0 \end{psmallmatrix}  &\text{for $n$ odd}.
\end{cases}
\end{equation*}

\vskip 2mm
 One can check by induction that $J_nJ_n^\intercal  = 1.$ It follows that $M^* =J_nM^\intercal  J_n^\intercal$ is an involution of $M_{2^n}(R).$ Importantly, their relation to Clifford algebras comes from the following equations :

\begin{equation}\label{eq1}
 \mathcal{S}_{n-1}^*(v,w)  = 
\begin{cases}
\mathcal{S}_{n-1}(v,w) &\text{ for $n$ odd,}\\\\
\overline{\mathcal{S}_{n-1}(v,w)} &\text{ for $n$ even.}
 \end{cases}
\end{equation}
We will sometimes omit the subscripts and simply write $J, \mathcal{S}$ or $\mathcal{S}(v,w)$ if the length of the vectors are clear.   
\vskip 1mm
The map
$\phi : H(R^n) \rightarrow M_{2^n}(R)$ defined by 
$\phi(v,w) =  \bigl(\begin{smallmatrix}
0                  & \mathcal{S}_{n-1}(v,w) \\
\overline{\mathcal{S}_{n-1}(v,w)}         & 0 \end{smallmatrix}\bigr)$ 
induces an $R$-algebra homomorphism $\phi : \Cl \rightarrow M_{2^n}(R).$
In fact $\phi$ is an isomorphism (see \cite[Section 3.1]{CV1}); the elements $\phi(v,w)$ give a set of generators of the Clifford algebra.
In addition, the involution $M^* =JM^\intercal  J^\intercal $ turns out be what is called the standard involution of the Clifford algebra \cite[Theorem 4.1]{CV1}.
Note that the quadratic form is $q(v,w) = \mathcal{S}(v,w) \overline{\mathcal{S}(v,w)}.$ For $\mathcal{S}_i = \mathcal{S}(v_i, w_i)$, the corresponding bilinear form is \[\langle \mathcal{S}_1, \mathcal{S}_2 \rangle = \mathcal{S}_1\overline{\mathcal{S}_2} + \mathcal{S}_2\overline{\mathcal{S}_1} = v_1 \cdot w_2^\intercal  + v_2 \cdot w_1^\intercal .\]

\vskip 5mm 
\subsection{Properties of the basis vectors}

Let $e_i$ and $f_i$ denote the standard basis vectors of $R^n$ and $R^{n^*}$ respectively. Define \[\mathcal{E}_i = \mathcal{S}_{n-1}(e_i, 0), \hskip 2mm  \mathcal{F}_i = \mathcal{S}_{n-1}(0,f_i).\]
Notice that $\mathcal{E}_1 = \bigl(\begin{smallmatrix}
1         & 0\\
0 & 0 \end{smallmatrix}\bigr)$ and $\mathcal{F}_1 = \bigl(\begin{smallmatrix}
0         & 0\\
0 & 1 \end{smallmatrix}\bigr).$ For $i >1$ the matrices $\mathcal{E}_i, \mathcal{F}_i$ are of the form $\bigl(\begin{smallmatrix}
0         & \mathcal{X}\\
-\overline{\mathcal{X}}     & 0 \end{smallmatrix}\bigr)$ for some Suslin matrix $\mathcal{X}$ with $ \mathcal{X}\overline{\mathcal{X}} =0.$

It is easy to check that the elements $\mathcal{E}_i$, $\mathcal{F}_i$ satisfy the following elementary properties.
\vskip 3mm 
\begin{lemma}\label{properties}
Let $\mathcal{X}_k \in \{\mathcal{E}_k, \mathcal{F}_k\}$ for $1 \leq k \leq n.$ Let $i \neq 1.$ Then
\begin{enumerate}[label={\alph*.}]
\item $\mathcal{X}_1^2 = \mathcal{X}_1$ and $\mathcal{X}_1+\overline{\mathcal{X}_1} = 1$ 
\item $\overline{\mathcal{X}_i} = -\mathcal{X}_i$ and $\mathcal{X}_i^2 = 0.$
\item $\mathcal{X}_i\mathcal{X}_1 =\overline{\mathcal{X}_1} \mathcal{X}_i.$
\end{enumerate}
\end{lemma}

\vskip 5mm
\begin{thm}\label{commutator}
Let $\mathcal{X}_k \in \{\mathcal{E}_k, \mathcal{F}_k\}$ for $1 \leq k \leq n.$  For any $\lambda \in R$ and $1< i,j \leq n$, we have the following commutator relations :
\[1 + \lambda \mathcal{X}_i\mathcal{X}_j = [1+ \lambda \mathcal{X}_i\mathcal{X}_1, 1 + \mathcal{X}_1\mathcal{X}_j] \] 
\end{thm}

\begin{proof}
It follows from Lemma $\ref{properties}$ that the inverse of $1 + \lambda \mathcal{X}_i\mathcal{X}_1$ is $1 - \lambda\mathcal{X}_i\mathcal{X}_1.$ Similarly the inverse of $1 + \mathcal{X}_1\mathcal{X}_j$ is $1 - \mathcal{X}_1\mathcal{X}_j.$ 
Moreover, since $\mathcal{X}_i^2 = \mathcal{X}_j^2 = 0$ and $\mathcal{X}_i\mathcal{X}_j  + \mathcal{X}_j\mathcal{X}_i = \langle \mathcal{X}_i, \mathcal{X}_j \rangle = 0$,  
any term where $\mathcal{X}_i$ or $\mathcal{X}_j$ appears twice is zero. Thus we are left with 
\begin{equation*}
\begin{split}
[1+ \lambda \mathcal{X}_i\mathcal{X}_1, 1 + \mathcal{X}_1\mathcal{X}_j]  
&=  1  - \lambda \mathcal{X}_1\mathcal{X}_j\mathcal{X}_i\mathcal{X}_1 + \lambda \mathcal{X}_i\mathcal{X}_1\mathcal{X}_j \\
&=  1 + \lambda \mathcal{X}_i\mathcal{X}_j (\mathcal{X}_1 + \overline{\mathcal{X}_1})
\end{split}
\end{equation*}
Since $\mathcal{X}_1 + \overline{\mathcal{X}_1} = 1$ we are done.
\end{proof}

%
%
%

\vskip 5mm
 \section{The Elementary Spin Group}
 
 \vskip 5mm 
  As stated earlier, the Clifford algebra is a $Z_2$-graded algebra $Cl = Cl_0 \oplus Cl_1.$ Under the isomorphism $\phi : Cl \cong M_{2^n}(R)$, the elements of $Cl_0$ correspond to matrices of the form 
$\bigl(\begin{smallmatrix}
g_1      &  0  \\
0        &  g_2 \end{smallmatrix}\bigr).$ 
\vskip 2mm
\noindent The Spin group is defined as 
    \[ \Spin_{2n}(R):= \{x \in Cl_0 \,| \, xx^* = 1 \text{ and } xH(R^n)x^{-1} = H(R^n)\}.\]  

\noindent  Just like we have the elementary group $E_n(R)$ corresponding to $SL_n(R)$, we have similar analogues for the orthogonal and Spin groups. 
 
 \begin{defn}
 \hfill
 \begin{enumerate}[label={\alph*.}]
\item  Let $e_{ij}$ denote the matrix with $1$ in the $(i,j)$ position and zeroes everywhere else.  For $i \neq j$, define 
 \[E_{ij}(\lambda) = 1 + \lambda e_{ij}\]
The matrices $E_{ij}(\lambda)$ are called elementary matrices and 
the group generated by $n \times n$ elementary matrices is called the elementary group $E_n(R).$
 \vskip 2mm
\item Let $\partial$ denote the permutation $(1 \ n+1)...(n \ 2n).$
We define for $1 \leq i \neq j \leq 2n$, $\lambda \in R$,
\[E^o_{ij}(\lambda) = I_{2n} + \lambda(e_{ij} - e_{\partial(j)\partial(i)}). \]
We call these the elementary orthogonal matrices
and the group generated by them is called the elementary orthogonal group $EO_{2n}(R).$
\vskip 2mm
\item From the definition of the Spin group, we have the map $\pi : Spin _{2n}(R) \rightarrow O_{2n}(R)$ given by
\[\pi(g) : (v,w) \rightarrow g\cdot (v,w) \cdot g^{-1} \text{ for $g \in \Spin_{2n}(R)$}.\]
We denote by $\Epin_{2n}(R)$ the inverse image of $EO_{2n}(R)$ under $\pi.$ 
\end{enumerate}
\end{defn}

\noindent The group $\Epin_{2n}(R)$ satisfies the following exact sequence (see \cite[p. 189]{B2})
\[ 1 \rightarrow \mu_2(R) \rightarrow \Epin_{2n}(R) \rightarrow EO_{2n}(R) \rightarrow 1\]
where $\mu_2(R) = \{x \in R : x^2 = 1 \}.$ 
\vskip 2mm 

Let $U_{2n-1}(R)$ denote the unit sphere in  $H(R^n).$
The action of $O_{2n}(R)$ on $U_{2n-1}(R)$ induces an action of $\Spin_{2n}(R)$  on  
$U_{2n-1}(R)$ (via $\pi$). Since $\pi : \Epin_{2n}(R) \rightarrow EO_{2n}(R)$ is surjective, it follows that 
 \begin{equation}\label{eq2}
 \frac{U_{2n-1}(R)}{\Epin_{2n}(R)} =  \frac{U_{2n-1}(R)}{EO_{2n}(R)},
\end{equation}

 \begin{lemma}\label{lemma1}
 There is a homomorphism $H : E_n(R) \rightarrow EO_{2n}(R)$ given by  $\varepsilon \rightarrow \left(\begin{smallmatrix}
\varepsilon     &  0 \\  
0     &  \varepsilon^{\intercal^{-1}} \end{smallmatrix}\right) \in EO_{2n}(R).$
 \end{lemma}
  \begin{proof}
 The lemma follows from the observation that $H(E_{ij}(\lambda)) = E^o_{ij}(\lambda). \qedhere$
 \end{proof}

\vskip 2mm

 \subsection{Generators of $\Epin_{2n}(R)$} 
 \hfill \vskip 1mm
 
 Let $V = R^n$ with standard basis $e_1,\cdots, e_n$ and dual basis $f_1,\cdots, f_n$ for $V^*.$ We will identify $H(V)$ with the corresponding matrices in the Clifford algebra. In terms of Suslin matrices,  
\[e_i =  \begin{bsmallmatrix}
0          & \mathcal{S}_{n-1}(e_i, 0)\\
\overline{\mathcal{S}_{n-1}(e_i, 0)}    & 0 \end{bsmallmatrix}, \hskip 5mm
f_i =  \begin{bsmallmatrix}
0          & \mathcal{S}_{n-1}(0, f_i)\\
\overline{\mathcal{S}_{n-1}(0, f_i)}    & 0 \end{bsmallmatrix} \]

\noindent It can be proved 
(see \cite[Section 4.3]{B2}) that $\Epin_{2n}(R)$ is generated by elements of the form $1 + \lambda e_ie_j,  1 + \lambda e_if_j, 1 + \lambda f_if_j$ with $\lambda \in R$,
$1 \leq i,\ j \leq n$, $i \neq j.$ 
\vskip 2mm

Let $(x_k, \mathcal{X}_k) \in \{ (e_k, \mathcal{E}_k), (f_k, \mathcal{F}_k) \}.$ Then the generator $1 + \lambda x_ix_j $ corresponds to the matrix 
\[ \phi(1 + \lambda x_ix_j ) = \begin{bsmallmatrix}
1 + \lambda \mathcal{X}_i\overline{\mathcal{X}_j}     & 0 \\
0    & 1 + \lambda  \overline{\mathcal{X}_i}\mathcal{X}_j\end{bsmallmatrix}\]

\noindent Since $e_i, e_1$ are orthogonal we have $e_ie_1 = - e_1e_i.$ Similarly $f_i f_1 = -f_1f_i.$ By also taking into account the commutator relations in Theorem \ref{commutator}, we find that $\Epin_{2n}(R)$ is generated by the (smaller) set of elements of the type  
\begin{equation}\label{eq3}
1 + \lambda e_1e_i,  \hskip 2mm 1+  \lambda e_1f_i,
\hskip 2mm 1 + \lambda f_1e_i, \hskip 2mm 1+ \lambda f_1f_i.   \hskip 6mm (1<i \leq n)
\end{equation}

Since $ \overline{\mathcal{X}_i} = -  \mathcal{X}_i$ for $i >1$, these generators correspond to matrices of the form 
\[\phi(1 +\lambda x_1x_i)  = \begin{bsmallmatrix}
1 - \lambda \mathcal{X}_1\mathcal{X}_i     & 0    \\
0    & 1 -  \lambda \bar{\mathcal{X}_1}\bar{\mathcal{X}_i}\end{bsmallmatrix}.\]

%
%
%
%

\vskip 2mm 
\subsection{The action of the Epin group.}\label{action} 
\hfill \vskip 1mm

So how do the generators of $\Epin_{2n}(R)$ act on the quadratic space? Let $\mathcal{X}_k \in \{\mathcal{E}_k, \mathcal{F}_k\}$ for $ 1 \leq k \leq n.$ 
\vskip 2mm
Let $g = \begin{bsmallmatrix}
1 - \lambda \mathcal{X}_1\mathcal{X}_i     & 0    \\
0    & 1 -  \lambda \bar{\mathcal{X}_1}\bar{\mathcal{X}_i}\end{bsmallmatrix}$ for some $i>1.$ Since $\overline{\mathcal{X}_i} = -  \mathcal{X}_i$ and 
 $\overline{\mathcal{X}_1}{\mathcal{X}_i} = \mathcal{X}_i\mathcal{X}_1$, we have 
\[g\begin{bsmallmatrix}
0                  & \mathcal{S}(v,w) \\
\overline{\mathcal{S}(v,w)}         & 0 \end{bsmallmatrix} g^{-1} = \begin{bsmallmatrix}
0                  & \mathcal{S}(v',w') \\
\overline{\mathcal{S}(v',w')}         & 0 \end{bsmallmatrix},\] where 
\begin{equation*}
\begin{split}
 \mathcal{S}(v',w') &= (1- \lambda \mathcal{X}_1 \mathcal{X}_i)  \cdot \mathcal{S}(v,w)  \cdot (1 + \lambda \bar{\mathcal{X}_1}\bar{\mathcal{X}_i})\\
&=  (1 - \lambda \mathcal{X}_1 \mathcal{X}_i)  \cdot \mathcal{S}(v,w)  \cdot (1 - \lambda \mathcal{X}_i\mathcal{X}_1)\end{split}
\end{equation*}

\noindent Recall that for $i >1$ the matrices $\mathcal{E}_i, \mathcal{F}_i$ are of the form $\bigl(\begin{smallmatrix}
0         & \mathcal{X}\\
-\overline{\mathcal{X}}     & 0 \end{smallmatrix}\bigr)$ for some Suslin matrix $\mathcal{X}$ with $ \mathcal{X}\overline{\mathcal{X}} =0.$
There are two cases :
 \begin{itemize} 
 \item $\mathcal{X}_1 = \mathcal{E}_1$ : Then $1 - \lambda \mathcal{E}_1\mathcal{X}_i$ and $1- \lambda \mathcal{X}_i\mathcal{E}_1$ will be equal to  $\bigl(\begin{smallmatrix}
1         & - \lambda \mathcal{X}\\
0 & 1 \end{smallmatrix}\bigr)$ and  $\bigl(\begin{smallmatrix}
1         & 0 \\
\lambda \overline{\mathcal{X} }   & 1 \end{smallmatrix}\bigr)$ respectively.
 
 \item $\mathcal{X}_1 = \mathcal{F}_1$ : Then $1 - \lambda \mathcal{F}_1\mathcal{X}_i$ and $1- \lambda \mathcal{X}_i\mathcal{F}_1$ will be equal to  
 $\bigl(\begin{smallmatrix}
1         & 0 \\
\lambda \overline{\mathcal{X} }   & 1 \end{smallmatrix}\bigr)$ and 
 $\bigl(\begin{smallmatrix}
1         &- \lambda \mathcal{X}\\
0 & 1 \end{smallmatrix}\bigr)$ respectively.
 \end{itemize}


\vskip 1mm 

The next two lemmas calculate  the action of the generators $g \in \Epin_{2n}(R).$ They imply that $g (v,w) g^{-1} = ( v \sigma,  w \sigma^{\intercal^{-1}})$  for some $\sigma \in E_n(R)$ (Theorem \ref{thm1}). This plays a crucial part in showing that there is a bijection between $\Epin_{2n}(R)$ 
orbits of the unit sphere and $E_n(R)$ orbits of unimodular rows for any $n\geq 3$ (Theorem \ref{main}).

\vskip 2mm

\begin{lemma}\label{lemma2}
Let $\mathcal{X}, \mathcal{T} \in M_{2^k}(R)$ be two Suslin matrices and $ \mathcal{X\overline{X}} =0.$ Let $\mathcal{S} = \begin{psmallmatrix}
a        &  \mathcal{T}\\
-\overline{ \mathcal{T}}     & b \end{psmallmatrix}.$ 
Then 
\[\begin{psmallmatrix}
1         &  \mathcal{X}\\
0       & 1 \end{psmallmatrix} 
 \mathcal{S}
 \begin{psmallmatrix}
1         & 0\\
-\overline{\mathcal{X} }& 1 \end{psmallmatrix}
 =  \begin{bsmallmatrix}
a -  \mathcal{\langle X,T  \rangle}      &  \mathcal{T} + b\mathcal{X}\\
\\
-\overline{ \mathcal{T}} - b\overline{\mathcal{X}}     & b \end{bsmallmatrix},\]

\[\begin{psmallmatrix}
1         & 0\\
-\overline{\mathcal{X} }      & 1 \end{psmallmatrix} 
\mathcal{S}
\begin{psmallmatrix}
1         & \mathcal{X} \\
0       & 1 \end{psmallmatrix}
 =  \begin{bsmallmatrix}
a         &  \mathcal{T} + a   \mathcal{X}\\
\\
 -\overline{ \mathcal{T}} -a\overline{\mathcal{X}}   & b - \mathcal{\langle X,T  \rangle} \end{bsmallmatrix}.\]
\end{lemma}
\begin{proof}
Let $\mathcal{Y} = \mathcal{X} + \mathcal{T}.$ Since $\mathcal{Y} \overline{\mathcal{Y}} =  \overline{\mathcal{Y}}\mathcal{Y}$, note that
\[ \mathcal{\langle X,T  \rangle =  
 X\overline{T} + T\overline{X} =
\overline{X} T + \overline{T}X =
\langle \overline{X}, \overline{T} \rangle}.\]  
The proof follows by straightforward matrix multiplication.
\end{proof}

\vskip 3mm

\begin{lemma}\label{lemma3}
Let $\mathcal{X} \in \{- \lambda \mathcal{E}_k, - \lambda \mathcal{F}_k\}$ where $1 \leq k \leq n$ and $\lambda \in R.$ Suppose $v \cdot w^\intercal  = 1$  for two vectors $v ,w \in R^{n+1}$ with $n \geq 2.$  Then
\[\begin{psmallmatrix}
1         & \mathcal{X}\\
0       & 1 \end{psmallmatrix} 
 \mathcal{S}(v,w)
 \begin{psmallmatrix}
1         & 0\\
-\overline{\mathcal{X}} & 1 \end{psmallmatrix}
= \mathcal{S}(v\varepsilon,w \varepsilon^{\intercal^{-1}})\] 

\[\begin{psmallmatrix}
1         & 0\\
-\overline{\mathcal{X} }      & 1 \end{psmallmatrix} 
\mathcal{S}(v,w)
\begin{psmallmatrix}
1         & \mathcal{X} \\
0       & 1 \end{psmallmatrix}
 =  \mathcal{S}( v \sigma,  w \sigma^{\intercal^{-1}})\]
 for some $\varepsilon, \sigma \in E_{n+1}(R).$ 
\end{lemma}

\begin{proof}
Let $\mathcal{X} = - \lambda \mathcal{E}_k.$ The proof is similar in the other case.
Write $v = (a_0,\cdots, a_n)$ and $w = (b_0, \cdots, b_n).$ 
\vskip 1mm
\noindent  From Lemma~\ref{lemma2}, we have
$\begin{psmallmatrix}
1         & \mathcal{X}\\
0       & 1 \end{psmallmatrix} 
 \mathcal{S}(v,w)
 \begin{psmallmatrix}
1         & 0\\
-\overline{\mathcal{X}} & 1 \end{psmallmatrix}
= \mathcal{S}(v', w') $,
where 
\[
\begin{bmatrix}
v'\\
w' \end{bmatrix} = 
\begin{bmatrix}
(a_0 +\lambda b_k, \cdots,a_k -\lambda b_0, \cdots ,a_n)  \\
w\end{bmatrix}.
\]
Since $v' \cdot w^\intercal  = v \cdot w^\intercal = 1$, it follows from \cite[Corollary 2.7]{S} that the matrices
\[\varepsilon = I_n + w^\intercal (v' - v) , \hskip 4mm (\varepsilon^\intercal )^{-1} = I_n - (v' - v)^\intercal  w\]
are in $E_n(R).$ We have   
$(v \varepsilon, w  \varepsilon^{\intercal^{-1}}) = (v' ,w).$ (The hypothesis $n \geq 2$ is needed to use \cite[Corollary 2.7]{S} where the length of the unimodular rows must be at least $3$).
\vskip 1mm 

For the second part, we use Lemma \ref{lemma2} and get
$\begin{psmallmatrix}
1         & 0\\
-\overline{\mathcal{X} }      & 1 \end{psmallmatrix} 
\mathcal{S}(v,w)
\begin{psmallmatrix}
1         & \mathcal{X} \\
0       & 1 \end{psmallmatrix}
 = \mathcal{S}(v'', w''), $
 where
\[
\begin{bmatrix}
v''\\
w'' \end{bmatrix} = 
\begin{bmatrix}
(a_0, \cdots,a_k - \lambda a_0, \cdots ,a_n)  \\
(b_0 +\lambda b_k, \cdots, b_n)\end{bmatrix}.
\]
Clearly $ (v'', w'') =
( v \sigma,  w \sigma^{\intercal^{-1}})$
where $\sigma = E_{1, k+1}(-\lambda).$
\end{proof}


%

\vskip 5mm
 \section{The bijection between $\Epin_{2n}(R)$ and $E_n(R)$ orbits}

We are now ready to prove the bijection between $E_n(R)$-orbits of unimodular rows and $\Epin_{2n}(R)$-orbits on the unit sphere in $H(R^n) = R^n \oplus {R^n}^*.$ We will break it down into simple parts with each part explaining one aspect of the bijection.

\begin{thm}\label{thm1}
Let $n\geq 3.$ Let $q(v,w) =1$ and 
and $g \in \Epin_{2n}(R).$ Then \[g (v,w) g^{-1} = ( v \sigma,  w \sigma^{\intercal^{-1}})\]  for some $\sigma \in E_n(R).$ 
\end{thm}
\begin{proof}
It is enough to prove the theorem for the generators of $\Epin_{2n}(R)$ (given in Section 3.3) and this was done in Section \ref{action} and Lemma \ref{lemma3}.
\end{proof}

\vskip 5mm

\begin{rmk}
There are two papers in the literature which prove some variation of the above theorem, though neither of them discuss Spin groups. The special case $n=3$ was considered in the proof of \cite[Corollary 7.4]{SV}, and an alternate approach can be found in \cite[Lemma 3.2]{JR1}. Both the papers study different group structures and connect them to the elementary-group actions on unimodular rows. In Part B, we will analyze the case $n=3$ and use it to explain the Vaserstein symbol.
\end{rmk}
\vskip 2mm

\begin{thm} \label{thm2}
Let  $n \geq 3.$ If $q(v, w_1) = q(v, w_2) = 1$, then $(v,w_1)$ and $(v,w_2)$ are in the same $EO_{2n}(R)$ and $\Epin_{2n}(R)$ orbits.
\end{thm}
\begin{proof}
By our hypothesis, we have $v\cdot w_1^\intercal  =  v\cdot w_2^\intercal = 1.$ Then it follows, from  \cite[Corollary 2.7]{S}, that the matrix
$$\varepsilon := I_n + v^\intercal (w_2 - w_1) \in E_n(R).$$  Since $w_1 \cdot \varepsilon = w_2$, both $w_1, w_2$ lie in the same $E_n(R)$ orbit.

\noindent By Lemma \ref{lemma1}, we have $H :  \varepsilon^{\intercal^{-1}} \rightarrow \left(\begin{smallmatrix}
\varepsilon^{\intercal^{-1}}     &  0 \\  
0     &  \varepsilon \end{smallmatrix}\right) \in EO_{2n}(R).$ Since $\varepsilon^{\intercal ^{-1}} = I_n - (w_2 - w_1)^\intercal  v$, it is easy to check that 
\[w_1 \varepsilon = w_2,\] 
\[v \varepsilon^{\intercal ^{-1}}  = v.\]

\noindent Therefore  $(v, w_1)$ and $(v,w_2)$ lie in same $EO_{2n}(R)$ orbit, and so by Equation (\ref{eq2}) in Section 3, they lie on the same $\Epin_{2n}(R)$ orbit.
\end{proof}

\vskip 2mm

Let $U_{2n-1}(R)$ be the unit sphere in $H(R^n).$ By the above theorem, the map $Um_n(R) \rightarrow \frac{U_{2n-1}(R)}{\Epin_{2n}(R)} $ given by $v \rightarrow (v,w)$ is well defined. 

\vskip 3mm

\begin{thm}\label{main}
Let $(v_1, w_1) , (v_2,w_2)$ be two points on the unit sphere $U_{2n-1}(R)$, where $n\geq 3.$ Then $(v_1,w_1) \widesim{\Epin_{2n}(R)}  (v_2,w_2)$ if and only if $v_1 \widesim{E_n(R)} v_2.$ 
\vskip 1mm
In other words, there is a bijection between the sets (of orbits) 
\[\frac{Um_n(R)}{E_n(R)} \longleftrightarrow \frac{U_{2n-1}(R)}{\Epin_{2n}(R)} =  \frac{U_{2n-1}(R)}{EO_{2n}(R)}. \qedhere \]
\end{thm}

\begin{proof}
 Suppose for the two unimodular rows $v_1, v_2$, we have $v_1 \cdot \varepsilon = v_2$ for some $\varepsilon \in E_n(R).$ 
 Then Theorem \ref{thm2} implies that
 \[ (v_1,w_1) \widesim{\Epin_{2n}(R)} (v_1 \cdot \varepsilon ,w_1 \cdot \varepsilon^{\intercal^{-1}}) \widesim{\Epin_{2n}(R)} (v_2, w_2).\] 
 
 \noindent  On the other hand, suppose $(v_1,w_1) \widesim{\Epin_{2n}(R)} (v_2,w_2).$ Then Theorem $\ref{thm1}$ implies that $v_1 \widesim{E_n(R)} v_2.$ 
Therefore we have a bijection
$\frac{Um_n(R)}{E_n(R)} \longleftrightarrow \frac{U_{2n-1}(R)}{\Epin_{2n}(R)}.\qedhere$
\end{proof}

\begin{cor}
Let $(v_1, w_1) , (v_2,w_2)$ be two points on the unit sphere $U_{2n-1}(R)$, where $n\geq 3.$ Then $(v_1,w_1) \widesim{EO_{2n}(R)}  (v_2,w_2)$ if and only if $v_1 \widesim{E_n(R)} v_2.$ 
\end{cor}

The above bijection says that for any $g \in \Epin_{2n}(R)$ and a point $(v,w)$ on the unit sphere, \[g (v,w) g^{-1} =( v \sigma,  w \sigma^{\intercal^{-1}})\]
for some $\sigma \in E_n(R).$ Here, the element $\sigma \in E_n(R)$ may vary with the choice of $(v,w).$ It should be stressed that the above bijection does not imply that the groups $E_n(R)$ 
and $\Epin_{2n}(R)$ are isomorphic. Only the corresponding orbit spaces are in bijection.



 \vskip 5mm
{\centering \Large \part{ Interpreting the Vaserstein symbol using Spin groups}}
 \vskip 5mm 
 In this part, we will focus on the case $n=3$ and examine the Vaserstein symbol.
 \vskip 5mm
\section{The Vaserstein symbol}

A matrix $A$ is said to be \emph{alternating} if $A^{\intercal} = -A$ and the diagonal elements of $A$ are all zero. It is well known that there exists a polynomial $pf$ called the Pfaffian (in the matrix elements with integral coefficients, depending only on the size of the matrix) such that $\det(A) = (pf(A))^2$ for all alternating matrices $A$ (see \cite[Ch. XV, \S 9]{L}). 

\begin{defn}
(\cite[p. 945]{SV}) The elementary symplectic-Witt group $W_E(R)$ is an abelian group consisting of (equivalent classes of) alternating matrices with Pfaffian $1.$
For alternating matrices $\alpha_r \in M_{r}(R)$ their sum is  defined as \[\alpha_r \perp \alpha_s : = \begin{psmallmatrix}
                 \alpha_r & 0\\
                 0 & \alpha_s\\
                \end{psmallmatrix} \in M_{r+s}(R).\]
 The identity element is $\psi_{r} = \psi_{r-1}\perp 
                \psi_{1}$ where $\psi_{1} = \begin{psmallmatrix}
                 0& 1\\
                -1 & 0\\
                \end{psmallmatrix} .$           
Two matrices $\alpha_r, \alpha_s$ are said to be equivalent if $\alpha_r \perp \psi_{s+l} = \varepsilon
(\alpha_s \perp \psi_{r+l})\varepsilon^\intercal ,$ for some $l \geq 0$ and  $\varepsilon \in  E(R).$
\end{defn}

The Vaserstein symbol is a map $\frac{Um_3(R)}{E_3(R)} \rightarrow W_E(R)$, giving a symplectic structure on orbits of unimodular rows. This is done by identifying a point on the unit sphere $(v,w)\in H(R^3)$ with a $4 \times 4$ alternating matrix.  Here we will use Suslin matrices which helps us to see the connection to Clifford algebras and Spin groups.

Let  $v=(a_1,a_2,a_3)$ and $w =(b_1,b_2,b_3).$ Recall from Section \ref{suslin} that 
\[ \mathcal{S}_2(v,w) = \begin{psmallmatrix}
a_1 & 0 & a_2 & a_3\\
0   & a_1 & -b_3 & b_2  \\
-b_2 & a_3 & b_1 & 0 \\
-b_3 & -a_2 & 0  & b_1 \end{psmallmatrix},
\hskip 4mm     J_2  = \begin{psmallmatrix}
0 & 1 &  0 & 0\\
-1  & 0 &  0& 0  \\
0 & 0 & 0 & -1 \\
0 & 0 & 1  & 0 \end{psmallmatrix}\]
and
$J\mathcal{S}^{\intercal} J^{\intercal} =  \mathcal{S}$ (from Equation (\ref{eq1})). Since $J^{-1} = J^{\intercal}= -J$, this can be rewritten as 
\[ ( \mathcal{S}J)^{\intercal} = - \mathcal{S}J.\]

\vskip 2mm

We need one more matrix 
$\beta =\begin{psmallmatrix}
1 & 0 &  0 & 0\\
0  & -1 &  0& 0  \\
0 & 0 & 0 & 1 \\
0 & 0 & -1  & 0 \end{psmallmatrix}.$
Note that $\beta^{\intercal} = \beta^{-1}.$

Define
 \[ V(v,w) :=   \beta\mathcal{S}_2(v,w) J_2 \beta^{\intercal}=  \begin{pmatrix}
0 & -a_1 & -a_2 & -a_3\\
a_1   & 0 & -b_3 & b_2  \\
a_2 & b_3 & 0 & -b_1 \\
a_3 & -b_2 & b_1  & 0 \end{pmatrix}.\] 
The matrix $V(v,w)$ is alternating with Pfaffian $v\cdot w^{\intercal} = a_1b_1 + a_2b_2 + a_3b_3.$ Thus when $v\cdot w^{\intercal} =1 $, the matrix $V(v,w)$ represents an element of $W_E(R).$ It is the same matrix corresponding to the Vaserstein symbol (in \cite[\S 5]{SV}).  In the next section we will break down the Vaserstein symbol into two parts and interpret it using Spin groups : 

\begin{enumerate}[label = {\alph*.}]
\item Let $A_4(R)$ denote the set of $4 \times 4$ alternating matrices with Pfaffian $1.$ First, we will show that there is a bijection 
$\frac{Um_3(R)}{E_3(R)} \leftrightarrow \frac{A_4(R)}{E_4(R)}.$
As we will see, this follows from the isomorphism $\Epin_6{R} \cong E_4(R)$ and then utilizing the results from Part A to get
\[\frac{Um_3(R)}{E_3(R)} \leftrightarrow \frac{U_5(R)}{\Epin_6(R)} \leftrightarrow \frac{A_4(R)}{E_4(R)}.\]
\item Then the obvious inclusion map gives us $ A_4(R) \rightarrow W_E(R)$, thus revealing the Witt-group structure on orbits of unimodular rows. 
\end{enumerate}

\begin{rmk}
The Vaserstein symbol was introduced in \cite[\S 5]{SV} to study orbits of unimodular rows. Suslin and Vaserstein studied the injectivity and surjectivity of the Vaserstein symbol and proved that it is a bijection if $R$ is a commutative noetherian ring of Krull dimension two (see \cite[Corollary 7.4]{SV}). The recent paper \cite{GRK} gives a survey of the non-injectivity of the Vaserstein symbol in dimension three.

\vskip 1mm
Aravind Asok and Jean Fasel have provided an interpretation of the Vaserstein symbol using $\mathbb{A}^1$-homotopy theory. The paper \cite{F2} explains this connection in detail (also see \cite[Theorem 4.3.1]{AF2}). Following \cite{F2},  let $k$ be a perfect field and $Q_5$ be the smooth affine quadric with $k[Q_5] = k[x_1, x_2,x_3, y_1, y_2, y_3] /\langle \sum x_iy_i -1\rangle.$ Let $[X,Y]_{\mathbb{A}^1}$ denote the set of all morphisms from $X$ to $Y$ in the unstable $\mathbb{A}^1$-homotopy category $\mathcal{H}(k)$ as defined by Morel-Voevodsky.

For any smooth affine $k$-scheme $X = Spec (R)$, there is a natural bijection
 $[X, Q_5]_{\mathbb{A}^1} = [X, \mathbb{A}^3 \setminus 0]_{\mathbb{A}^1} = Um_3(R)/E_3(R).$ Moreover $Q_5$ is isomorphic to the quotient of algebraic varieties $SL_4/Sp_4$ giving us the composite map $Q_5 \rightarrow SL_4/Sp_4 \rightarrow SL/Sp.$ It turns out that the quotient $SL/Sp$ represents the (reduced) higher Grothendieck-Witt group $GW_{1,red}^3(X)$ which coincides with $W_E(R)$ for any smooth affine variety $X = Spec (R).$ Thus one has the following interpretation of the Vaserstein symbol 
 \[    Um_3(R)/E_3(R) =    [X, Q_5]_{\mathbb{A}^1} \rightarrow [X, SL/Sp]_{\mathbb{A}^1}  =W_E(R). \] 
\end{rmk}

%
%

\vskip 5mm
\section{The dictionary between Vaserstein and Suslin matrices}

\vskip 3mm 

We will borrow results from \cite{CV1} on the connection between Suslin matrices and Clifford algebras. Specifically we need the well-known exceptional isomorphisms $\Spin_6(R) \cong SL_4(R)$ and $\Epin_6(R) \cong E_4(R).$ (For a proof using Suslin matrices, see \cite[Theorems 7.1, 8.4]{CV1}. For another proof, see \cite[Ch. 5, \S 5.6]{Kn}).

Define $*$ to be the involution on $M_4(R)$ given by $ M^* = JM^{\intercal }J^{\intercal }$ where 
 $ J = \begin{psmallmatrix}
0 & 1 &  0 & 0\\
-1  & 0 &  0& 0  \\
0 & 0 & 0 & -1 \\
0 & 0 & 1  & 0 \end{psmallmatrix}.$
Note that $ *$ is an involution because $J^{\intercal } = -J =  J^{-1}.$  
\vskip 2mm
Let's identify the Suslin matrix $\mathcal{S}(v,w)$ with the element $(v,w)$ in the quadratic space $H(R^3).$ 
Under the isomorphism $\psi :\Spin_6(R) \cong SL_4(R)$, the Spin group behaves as  follows :  for $g \in SL_4(R)$,  the action is given by $g \bullet \mathcal{S}(v,w) = g\mathcal{S}(v,w)g^*.$ Simplifying the notation, we will sometimes write $\mathcal{S}, V$ instead of  $\mathcal{S} (v,w),  V(v,w).$

 \vskip 2mm 
 
Any $4 \times 4$ alternating matrix is of the form $V(v,w)$, corresponding to the element $(v,w ) \in H(R^3).$ Let $A_4(R)$ denote the set of all such matrices with $v\cdot w^{\intercal} =1$ (the unit sphere in $H(R^3)$). 
The group $SL_4(R)$ acts on the matrices $V(v,w)$ as $(g, V) \rightarrow gVg^T.$ Recall that the unit sphere in $H(R^3)$ is denoted by $U_5(R).$

\begin{thm}\label{vaserstein}
We have the bijection  
$ \frac{U_5(R)}{\Spin_6(R)} \leftrightarrow \frac{A_4(R)}{SL_4(R)}.$
Therefore, 
\[\frac{Um_3(R)}{E_3(R)} \leftrightarrow  \frac{U_5(R)}{EO_6(R)} = \frac{U_5(R)}{\Epin_6(R)} \leftrightarrow \frac{A_4(R)}{E_4(R)}\] 
where
$ v \rightarrow \mathcal{S}(v,w) \rightarrow V(v,w)$ for any element $(v,w) \in U_5(R).$ \end{thm}
\begin{proof}
The bijection between the $E_3(R)$-orbits of unimodular rows and $\Epin_6(R)$-orbits on the unit sphere follows from Theorem $\ref{main}$ in Part A. For the second part, note that $A_4(R)$ corresponds to the unit sphere in $H(R^3).$ We will now show that the group actions on $H(R^3)$ are the same.
Remember that  $V = \beta\mathcal{S}J\beta^{\intercal}.$  Let $ \mathcal{S}(v',w') = g \bullet \mathcal{S}(v,w).$ Writing $\mathcal{S} = \mathcal{S}(v,w)$, we have 
\begin{equation*}
\begin{split}
V(v', w') &= \beta(g \bullet \mathcal{S})J\beta^{\intercal} \\
&= \beta(g\mathcal{S}g^*)J\beta^{\intercal}\\
&=  \beta(g\mathcal{S}Jg^{\intercal }J^{\intercal }) J \beta^{\intercal} \\
&=  \beta g\mathcal{S}Jg^{\intercal } \beta^{\intercal}     \hskip 32mm  (\text{as } J^{\intercal} J = 1)\\
&=   (\beta g\beta^{\intercal}) (\beta \mathcal{S}J\beta^{\intercal })(\beta g^{\intercal} \beta^{\intercal})  \hskip 10mm  (\text{as }\beta^{\intercal} \beta = 1)  \\
&=  g'(\beta \mathcal{S}J\beta^{\intercal } ) g'^{\intercal}    \hskip 27mm     (g' = \beta g \beta{^\intercal})\\
&= g'V(v,w)g'^{\intercal}
\end{split}
\end{equation*}
where $ g' = \beta g \beta{^\intercal} = \beta g \beta^{-1}.$  
Since $E_n(R)$ is a normal subgroup of $GL_n(R)$ for $n\geq3$ (see \cite[Corollary 1.4]{S2}), it follows that $g' \in E_4(R)$ whenever $g \in E_4(R).$ Therefore we have the bijection between the respective (elementary) orbit spaces 
\[\frac{U_5(R)}{\Epin_6(R)} \leftrightarrow \frac{A_4(R)}{E_4(R)}.  \qedhere \]
\end{proof}

The above correspondence gives another proof of the following well-known exceptional isomorphism.

\begin{thm}
\[\Spin_5(R) \cong Sp_4(R). \]
\end{thm}
\begin{proof}
The proof follows by identifying $\Spin_5(R)$ as a subgroup of $\Spin_6(R)$ which fixes $(v,w) = (1,0, 0, 1, 0, 0).$ 
The elements of $\Spin_5(R)$ then correspond to matrices $g \in SL_4(R)$ such that $gg^* = 1.$ In other words,
$ g J g^{\intercal} = J.$ Since $J$ is (canonically) isometric to $\psi_2$, it follows that $Spin_5(R) \cong Sp_4(R).$
\end{proof}

In light of Theorem \ref{vaserstein}, the Vaserstein symbol $V : \frac{Um_3(R)}{E_3(R)} \rightarrow W_E(R)$ can now be decomposed as 
\[V : \frac{Um_3(R)}{E_3(R)} \cong  \frac{A_4(R)}{E_4(R)} \rightarrow W_E(R).\]

\noindent The injectivity (surjectivity) of the Vaserstein symbol boils down to the injectivity (surjectivity) of the map $\frac{A_4(R)}{E_4(R)} \rightarrow W_E(R)$, which  is defined naturally via the inclusion map. The interpretation in terms of Spin groups is summarized in the table below :
\vskip 5mm
\begin{tabular}{ c|c } 
 \hline
 Vaserstein symbol & The Spin group interpretation\\
 \hline
 \\
 $4 \times 4$ Vaserstein matrix $V(v,w)$ & Suslin matrix $\mathcal{S}_2(v,w)$, ($V =\beta\mathcal{S}_2J_2\beta^{\intercal}$) \\\\
 Action of $SL_4(R), E_4(R)$:  & Action of $\Spin_6(R), \Epin_6(R)$:   \\
 $(g ,V) \rightarrow gV g^{\intercal}$ &  $g \bullet \mathcal{S} = g\mathcal{S} g^*$ \\\\
 $A_4(R)$ & $U_5(R)$\\\\

 Orbits of Unimodular rows & Orbits on the sphere  ($v\cdot w^\intercal =1$) \\
 $ \frac{Um_3(R)}{E_3(R)} \cong \frac{A_4(R)}{E_4(R)}$ &  $\frac{U_5(R)}{\Epin_6(R)}$
 \\\\
 $SL_4(R), E_4(R)$ & $\Spin_6(R), \Epin_6(R)$\\\\
 $Sp_4(R)$ & $\Spin_5(R)$\\
 \hline
\end{tabular}
\vskip 5mm


\subsection{A question about $K\Spin_1(R)$} 
\hfill \vskip 1mm

It turns out that if $\mathcal{X,Y} \in M_{2^k}(R)$ are two Suslin matrices, then $ \mathcal{XYX}$ is also a Suslin matrix and $\mathcal{\overline{XYX} = \bar{X} \bar{Y} \bar{X}}$ (see \cite[Theorem 3.4]{CV1}). 
Now suppose $\mathcal{X \overline{X}} = 1$, and take $g = \bigl(\begin{smallmatrix}
\mathcal{X}                &  0 \\
0& \mathcal{\overline{X}}         \end{smallmatrix}\bigr).$ Then $g^{-1} = \bigl(\begin{smallmatrix}
\mathcal{\overline{X}}                         &  0 \\
0& \mathcal{X} \end{smallmatrix}\bigr)$ and
\begin{equation}\label{eq4}
g\bigl(\begin{smallmatrix}
0               & \mathcal{ Y} \\
\mathcal{\overline{Y}} &0         \end{smallmatrix}\bigr)
g^{-1}= \bigl(\begin{smallmatrix}
0               &  \mathcal{XYX} \\
\mathcal{\overline{XYX}} &0         \end{smallmatrix}\bigr).
\end{equation}

Moreover, as $\mathcal{X\overline{X}} =1$, we have $\mathcal{XYX \cdot \overline{XYX} = Y\overline Y}.$ Recall that an element $(v',w') \in H(R^n)$ corresponds to the matrix 
$\phi(v',w') =  \bigl(\begin{smallmatrix}
0                  & \mathcal{Y'} \\
\overline{\mathcal{Y'}}         & 0 \end{smallmatrix}\bigr)$, where $\mathcal{Y'} = \mathcal{S}_{n-1}(v',w').$
Therefore it follows from Equation (\ref{eq4}) that there is a map $U_{2n-1}(R) \rightarrow \Spin_{2n}(R)$ given by $(v,w) \rightarrow \bigl(\begin{smallmatrix}
\mathcal{S}_{n-1}(v,w)                &  0 \\
0& \overline{\mathcal{S}_{n-1}(v,w)}         \end{smallmatrix}\bigr).$ 

This induces a map $ \frac{U_{2n-1}(R)}{\Epin_{2n}(R)} \rightarrow \frac{\Spin_{2n}{R}}{\Epin_{2n}(R)} \rightarrow K\Spin_1(R).$ What is the relation between $W_E(R)$ and the abelian group $K\Spin_1(R)$?

\subsection{Vaserstein composition}

The paper \cite{SV} also introduced a composition law on unimodular rows (\cite[Theorem 5.2]{SV}). 
 The composition law was later generalized to $Um_n(R)$ by W. van der Kallen using what are called weak Mennicke symbols as follows (see \cite[Lemma 3.4]{vdk2}) : 

  Let $v_1 = (a_1, a_2, a_3, \cdots, a_n)$ and $v_2 = (c_1, c_2, a_3, \cdots, a_n)$ be two unimodular rows and choose $d_1, d_2$ such that the determinant of 
$\beta = \begin{psmallmatrix}
                 c_1 & c_2\\
                - d_2 & d_1\\
                \end{psmallmatrix}$  has image $1$ in 
                $R/\langle a_3, \cdots, a_n \rangle.$ Then
    \[ wms(v_1)wms(v_2) = wms (p, q, a_3, \cdots a_n)\]
  where  $(p,q) = (a_1, a_2) \beta .$
  
In Part C, we will introduce a new composition law on certain subspaces of $H(R^n)$ satisfying the same properties. (See Remark \ref{rmk1}). Moreover this law has the nice feature that it is expressed recursively using matrices. It turns out that this composition of unimodular rows is a special case of a more general law, which acts on certain subspaces of 
$A \oplus H(R^n)$ where $A$ is a composition algebra. The Vaserstein composition corresponds to the case where the composition algebra is the algebra of split quaternions.

As an illustration of the results in Part C, we will now interpret Vaserstein's composition rule using Suslin matrices for the case $n=3.$
\vskip 1mm
Fix $a_1 \in R$. Let $v = (a_1, a_2,a_3)$ and $v'= (a_1, a_2', a_3')$ be two unimodular rows (with $a_1 = a_1'$). Suppose $v \cdot w^{\intercal} = v' \cdot w'^{\intercal} = 1$ for some 
$w = (b_1, b_2,b_3)$ and $w' = (b_1', b_2', b_3').$ The Suslin matrices corresponding to $(v,w)$ and $(v',w')$ are 
 \[\mathcal{X} = \mathcal{S}(v,w) = \begin{psmallmatrix}
a_1     &  \alpha \\  
-\overline{\alpha}       &  b_1\end{psmallmatrix} \hskip 3mm \text{ and } \hskip 3mm 
\mathcal{Y} =\mathcal{S}(v',w')  = \begin{psmallmatrix}
a_1   &  \beta \\   
-\overline{\beta}       &  b_1'\end{psmallmatrix},\] where 
$\alpha = \begin{psmallmatrix}
                 a_2 & a_3\\
                - b_3 & b_2\\
                \end{psmallmatrix}$ and 
   $\beta = \begin{psmallmatrix}
                 a_2' & a_3'\\
                - b_3' & b_2'\\
                \end{psmallmatrix}.$
                
\vskip 1mm                
Finally define the composition
   \[\mathcal{X} \odot \mathcal{Y}
:=
\begin{pmatrix}
a_1     &  \alpha \beta \\\\    
-\overline{\alpha \beta}       &    \hskip 5mm b_1 + b_1' - a_1b_1 b_1'\end{pmatrix}\]
The element $\mathcal{X} \odot \mathcal{Y}$ is also a Suslin matrix and $q(\mathcal{X} \odot \mathcal{Y}) = q(\mathcal{X})q(\mathcal{Y}).$ Notice that the product
$\mathcal{X}\odot  \mathcal{Y} = \mathcal{S}(v'', w'')$ is similar to the Vaserstein composition, as we get $v'' = (a_1, (a_2, a_3)\beta).$

%

\vskip 5mm

{\centering\Large \part{A general Composition law}}
\vskip 5mm

\section{Preliminaries on Composition algebras}

\subsection{Notation} 
Let $R$ be a commutative ring and $V$ be a free $R$-module equipped with a quadratic form $q$ and basis $\{v_1,\cdots v_n\}.$ Let $B = \bigl(\langle v_i,v_j\rangle \bigr)$ be the matrix corresponding to the bilinear form $\langle v,w\rangle = q(v+w) - q(v) - q(w)$, for $v,w \in V.$

\vskip 1mm

 A quadratic form $q$ is said to be \emph{non-degenerate} if $\langle x,v\rangle =0 $ for all $v \in V$
implies that $x=0.$
Writing $x=(x_1,\cdots,x_n),$ this means that the equation $xB =0$ has only the trivial solution, or equivalently, $\det(B)$ is a non-zero divisor
\cite[Corollary 1.30]{MD}. 
We say that $q$ is \emph{non-singular} 
if $B$ is invertible.

One advantage with non-singular quadratic spaces $(V,q)$ is that the Clifford algebra $Cl(V,q)$ will have the structure of a graded Azumaya algebra \cite[Corollary 3.7.5]{HM} and this can be a useful tool in proving results. In this paper, we are only assuming that the quadratic form is non-degenerate. 
\begin{rmk}
For general theory of quadratic forms and Clifford algebras over commutative rings, see \cite{B2, HM, Kn}. The book \cite{HM} uses a slightly different terminology by referring to $q$ as weakly non-degenerate if $\det(B)$ is a non-zero divisor, and non-degenerate if $B$ is invertible.
\end{rmk}

\vskip 1mm
\subsection{Composition algebras}
A composition algebra $(A,q)$ over $R$ is a (not necessarily associative) $R$-algebra such that ($A$ is a free $R$-module of finite rank and)
\begin{itemize}
\item $A$ has an identity element and contains a copy of $R$, i. e., $R \hookrightarrow A$ (mapping $1\in R$ to $1 \in A$).
\item $A$ is equipped with a non-degenerate quadratic form $q$ satisfying 
$q(xy) = q(x)q(y)$
for all $x, y \in A.$
\end{itemize}
 For any composition algebra 
$(A,q)$, there is an involution $\alpha \rightarrow \overline{\alpha}$ such that 
$q(\alpha) = \alpha\bar{\alpha} = \bar{\alpha}\alpha = q(\overline{\alpha})$, for all $\alpha \in A.$  
The quadratic form on $A$ is sometimes referred to as the norm of $A$. It is known that $\rank(A)$ has to be $1, 2, 4,$ or $8$ (see \cite[V. 7.1.6]{Kn}).

Composition algebras of rank $4$ and $8$ are called \emph{quaternions} and \emph{octonions} respectively. Most of their applications arise when the base ring is a field. The new book \cite{Vo} captures some of the wide-ranging connections of quaternions to different branches of Mathematics. Some applications of octonions can be found in \cite{Ba,CS}. 

When the base ring is a field,  a scalar multiple of the norm of the composition algebra $A$ determines $A$ upto isomorphism \cite[\S 1.7]{SVe}. For non-singular quadratic forms, this result has been extended to quaternion algebras over commutative rings by Knus-Ojanguren-Sridharan \cite[Prop 4.4]{KOS}.  On the other hand, octonion algebras (over commutative rings) are not determined by their norm, and a counterexample has been provided by P. Gille \cite{G}.

\vskip 3mm

\section{Composition law}
\vskip 2mm

The following construction is inspired by the construction of Suslin matrices. 
Let $(A,q)$ be any composition algebra. Consider the quadratic space $A \oplus H(R)$, where $H(R)$ is a hyperbolic plane. 
For each element  $ (\alpha,a,b ) \in A \oplus H(R)$, the quadratic form is given by $$q(\alpha,a,b) = \alpha\bar{\alpha} +ab.$$

One can represent $(\alpha,a,b)$ as a matrix  $Z = \left(\begin{smallmatrix}
a     &  \alpha \\  
-\overline{\alpha}       &  b\end{smallmatrix}\right).$ Define $\overline{Z} =  \left(\begin{smallmatrix}
b     &  -\alpha \\  
\overline{\alpha}       &  a\end{smallmatrix}\right).$ Then \[q(Z) = Z\overline{Z} = \overline{Z}Z = \alpha\bar{\alpha} +ab.\]
\vskip 2mm 

For any such matrix, we will sometimes write $q_Z$ instead of $q(Z).$ One of the reasons we rewrite the elements as $2 \times 2$ matrices is that it is easier to express the composition law and generalize the analysis to $A \oplus H(R^n).$ In addition, as we shall see later, this matrix representation gives a simple description of the corresponding Clifford algebra.

\vskip 2mm
\subsection{Composition law for hyperplanes of $A \oplus H(R)$}
\hfill
\vskip 1mm

Let $X = \begin{pmatrix}
a     &  \alpha \\  
-\overline{\alpha}       &  b\end{pmatrix}$ and 
$Y = \begin{pmatrix}
a    &  \beta \\   
-\overline{\beta}       &  b'\end{pmatrix}.$
When $q_X = q_Y =1$, define 
\[X \odot Y
:=
\begin{pmatrix}
a   &  \alpha \beta \\\\    
-\overline{\alpha \beta}        &    \hskip 5mm  b +b' -abb' \end{pmatrix}\]
Then $q(X \odot Y) = 1.$
\vskip 2mm 

\noindent  For general $X, Y$ define
\[ X \odot Y
:=
\begin{pmatrix}
a      &  \alpha \beta \\\\    
-\overline{\alpha \beta}        &    \hskip 5mm bq_Y + b'q_X - abb'\end{pmatrix}.\]

\noindent From the equations
\[ \alpha \bar{\alpha} = q_X -ab, \hskip 5mm \beta \bar{\beta} = q_Y - ab',\]
it follows that \[ q(X \odot Y) = q(X) q(Y). \]
\noindent When the underlying composition algebra is associative, the operation $\odot$ is  also associative with the identity element $\left(\begin{smallmatrix}
a    &  1 \\\\    
-1       &  0\end{smallmatrix}\right).$ 
\vskip 2mm

If we take $a= b= b'=0$, then $X \odot Y
=
\left(\begin{smallmatrix}
0    &  \alpha \beta \\\\    
-\overline{\alpha \beta}       &  0\end{smallmatrix}\right)$ corresponds to the multiplication in the composition algebra.
When $A \cong M_2(R)$ is the algebra of split quaternions, then the above composition law gives us the Vaserstein composition on unimodular rows  stated in Part B (see Remark \ref{rmk1}).

\vskip 2mm

\subsection{The quadratic space $A \oplus H(R^n)$}
\hfill 
\vskip 1mm

Consider next the quadratic space $A \oplus H(R^n)$, where $H(R^n) = R^n \oplus {R^n}^*.$ For each element  $ (\alpha,v,w ) \in A \oplus H(R^n)$, the quadratic form is given by $$q(\alpha,v,w) = \alpha \bar{\alpha} +v\cdot w^\intercal .$$
Here $\alpha$ is an element of the composition algebra $A$ and $v,w \in R^n.$
\vskip 1mm
\noindent By fixing a basis of $R^n$, let us write $v =(a_1, \cdots, a_n)$ and $w =(b_1, \cdots, b_n).$
Let  $Z_1(\alpha,v,w) = \left(\begin{smallmatrix}
a_1     &  \alpha \\  
-\overline{\alpha}       &  b_1\end{smallmatrix}\right)$ and $\overline{Z_1(\alpha,v,w)} =  \left(\begin{smallmatrix}
b_1    &  -\alpha \\  
\overline{\alpha}       &  a_1\end{smallmatrix}\right).$ 
\vskip 1mm

\noindent For $i>1$, define recursively the matrices $Z_i (\alpha,v,w) := \left(\begin{smallmatrix}
a_i     &  Z_ {i-1}\\  
-\overline{Z_{i-1}}       &  b_i\end{smallmatrix}\right)$ and $\overline{Z_i(\alpha,v,w)} :=  \left(\begin{smallmatrix}
b_i    &  -Z_{i-1} \\  
\overline{Z_{i-1}}       &  a_i\end{smallmatrix}\right).$

\vskip 1mm

\noindent Then $Z_i$ is a $2^i \times 2^i$ matrix and \[ q(Z_i) = Z_i\bar{Z_i} = \bar{Z_i}Z_i = \alpha \bar{\alpha} +a_1b_1 +\cdots+ a_ib_i. \]

\vskip 3mm 
\subsection{Composition law for certain subspaces of $A \oplus H(R^n)$}
\hfill 
\vskip 1mm

Fix $v =(a_1,\cdots, a_n) \in R^n.$
\vskip 1mm
\noindent Let $\alpha,\beta \in A, w = (b_1,\cdots, b_n)$ and $w' =(b'_1,\cdots ,b'_n).$
 \vskip 2mm
 \noindent Write $X_i = Z_i(\alpha,v,w)$ and $Y_i = Z_i(\beta,v,w').$
By definition, we have 
\[q_{X_i} = a_ib_i +q_{X_{i-1}} \hskip 3mm \text{and} \hskip 3mm q_{Y_i} = a_ib'_i +q_{Y_{i-1}}.\]

\vskip 2mm
\noindent Define the composition $X_i \odot Y_i$ recursively as
\[X_i \odot Y_i
:=
\begin{pmatrix}
  a_i  &  X_{i-1}\odot Y_{i-1} \\\\    
-\overline{  X_{i-1}\odot Y_{i-1} }         &   \hskip 5mm  b_i q_{Y_i} +b'_iq_{X_i} -a_ib_ib'_i\end{pmatrix}.\]
\noindent By induction, it follows that 
\[q_{X_n \odot Y_n}  =q_{X_n}q_{Y_n}.\]

\begin{rmk}\label{rmk1}
As promised in Section 6.4, we will now interpret (van der Kallen's) composition of unimodular rows (see \cite[Lemma 3.4]{vdk2}) using $\odot.$

When $A \cong M_2(R) $ is the algebra of split quaternions, the matrices $Z(\alpha, v,w)$ are Suslin matrices. 
Let $v_1 = (a_1, a_2, a_3, \cdots, a_n)$ and $v_2 = (c_1, c_2, a_3, \cdots, a_n)$ be two unimodular rows  such that $v_i \cdot w_i^{\intercal} =1.$  
Suppose $\mathcal{S}(v_1, w_1) \odot \mathcal{S}(v_2, w_2)  = \mathcal{S}(v_3, w_3).$ Then \[v_3 = (p, q, a_3, \cdots , a_n)\] where $(p,q) = (a_1, a_2) \beta.$ Here 
 $\beta = \begin{psmallmatrix}
                 c_1 & c_2\\
                - d_2 & d_1\\
                \end{psmallmatrix}$
where $w_2 = (d_1, d_2, \cdots, d_n).$
Clearly the determinant of $\beta$  has image $1$ in 
                $R/\langle a_3, \cdots, a_n \rangle$ as $v_2 \cdot w_2^{\intercal} =1.$ 
\end{rmk}

\begin{rmk}\label{rmk2}
Suppose $R$ is a commutative Noetherian ring of stable dimension $d$ with $d \leq 2n-3.$ Let $v_1, v_2 \in Um_n(R)$ such that $v_i \cdot w_i ^\intercal=1.$ Then the Mennicke-Newman lemma says that one can find $ \varepsilon_i \in E_n(R)$ such that $v _i\varepsilon_i = (x_i, a_2, \cdots, a_n)$ (for some $x_i, a_k \in R$. See \cite[Lemma 3.2]{vdk3}).

Together with Theorem \ref{main}, this means that $(v_1,w_1), (v_2, w_2)$ are $\Epin_{2n}(R)$-equivalent to the points $(v_i \varepsilon_i ,w_i \varepsilon_i^{\intercal^{-1}})$ which lie on the same $(n+1)$-dimensional subspace (determined by the  elements $a_k$) where the composition $\odot$ can operate.
\end{rmk}

%
\vskip 5mm

\section{The Clifford algebra of $A \oplus H(R^n)$ : the quaternion case }

Here we will consider the case when $A$ is a quaternion algebra over $R.$ 
Let $V = A \oplus H(R^n)$ and we will continue representing its elements $(\alpha,v,w)$ as a matrix $Z_n(\alpha,v,w).$ Notice that 
$Z_n(\alpha,v,w) \in M_{2^n}(A).$
\vskip 2mm
Consider the map $\phi : V  \rightarrow  M_{2^{n+1}}(A)$ given by 
\[
(\alpha,v,w) \rightarrow
\begin{bmatrix}
0 & Z_n(\alpha,v,w)\\
\overline{Z_n(\alpha,v,w) }& 0\\
\end{bmatrix}
\]
Since $\phi(\alpha,v,w)^2 = q(\alpha,v,w)$, by the universal property of Clifford algebras the map lifts to an $R$-algebra homomorphism 
\[\phi: Cl(V) \rightarrow M_{2^{n+1}}(A). \]
The map $\phi$ is in fact a $\mathbb{Z}_2$-graded homomorphism, where the even and odd elements of $M_{2^{n+1}}(A)$ correspond to matrices of the form
$(\begin{smallmatrix}
*    &  0  \\
0       &*  \end{smallmatrix})$ and $(\begin{smallmatrix}
 0   &  *  \\
*        & 0 \end{smallmatrix}).$

\begin{thm}\label{iso}
The map $\phi: Cl(V) \rightarrow M_{2^{n+1}}(A)$ is injective. Moreover, $\rank[Cl(V)] = \rank[M_{2^{n+1}}(A)].$
\end{thm}
\begin{proof}
Let $ker(\phi)$ denote the kernel of $\phi.$ Since $\phi$ restricts to an injective map on $V$ and $R \hookrightarrow A$, we have $ker(\phi) \cap R = \{0\}.$ 
Then it follows from \cite[Theorem 2.7]{CV2} that $\phi$ is injective.

Since $M_{2^{n+1}}(A) = M_{2^{n+1}}(R) \otimes A$, one can see that its rank is $2^{2n+4}$, the same as $\rank(Cl) = 2^{\rank(V)}.$ \end{proof}

\begin{rmk}
If the quadratic space $(A,q_A)$ (and thus $(V,q)$) is non-singular, then both $M_{2^{n+1}}(A)$ and the Clifford algebra $Cl(V)$ are graded Azumaya algebras of the same rank (see \cite[Theorem 3.6.8, Corollary 3.7.5]{HM}). In that case, $\phi: Cl(V) \rightarrow M_{2^{n+1}}(A)$ will be an isomorphism (see \cite[Lemma 6.7.10]{HM}).
\end{rmk}

\section{Clifford algebra: the octonion case}\label{octonioncl}

\subsection{The embedding in the endomorphism ring}
Let $O$ be an octonion algebra over $R$.
The problem here is that the matrix algebra $M_{2^{n+1}}(O)$ is not associative anymore. However the octonion algebra $O$ has the interesting property that $ \overline{\alpha}(\alpha \beta) = \alpha(\overline{\alpha} \beta) = q(\alpha) \beta$ for all $\alpha, \beta \in O.$ (See \cite[Ch. V, \S 7]{Kn}).

Putting it another way, consider the left multiplication map $L : O \rightarrow \End(O)$
where $L_{\alpha}$ is left-multiplication by $\alpha.$ Since $q(\alpha) = q(\overline{\alpha}),$ these maps 
satisfy the property that $L_{\alpha}L_{\bar{\alpha}} = L_{\bar{\alpha}} L_{\alpha} = L_{q(\alpha)}.$
We will modify the matrices $Z_i(\alpha, v, w)$ by replacing $\alpha$ with $L_{\alpha}$ in the matrix.
\vskip 1mm

\noindent Define  $Z'_1(\alpha,v,w) = \left(\begin{smallmatrix}
a_1     &  L_{\alpha} \\  
-L_{\overline{\alpha}}      &  b_1\end{smallmatrix}\right)$ and $\overline{Z'_1(\alpha,v,w)} =  \begin{psmallmatrix}
b_1    &  -L_{\alpha} \\  
L_{\overline{\alpha}}       &  a_1\end{psmallmatrix}.$ 
\vskip 2mm

\noindent For $i>1$, define recursively the matrices $Z'_i (\alpha,v,w) := \left(\begin{smallmatrix}
a_i     &  Z'_ {i-1}\\  
-\overline{Z'_{i-1}}       &  b_i\end{smallmatrix}\right)$ and $\overline{Z'_i(\alpha,v,w)} :=  \left(\begin{smallmatrix}
b_i    &  -Z'_{i-1} \\  
\overline{Z'_{i-1}}       &  a_i\end{smallmatrix}\right).$

\subsection{The Clifford algebra}

We have the map
$\phi : Cl(V)  \rightarrow  M_{2^{n+1}}(\End(O))$ given by 
\[
(\alpha,v,w) \rightarrow
\begin{bmatrix}
0 & Z'_n(\alpha, v, w)\\
\overline{Z'_n(\alpha,v,w) }& 0\\
\end{bmatrix}
\]

\begin{thm}
The map $\phi: Cl(V) \rightarrow M_{2^{n+1}}(\End(O))$ is injective. Moreover, $\rank[Cl(V)] = \rank[M_{2^{n+1}}(\End(O))].$
\end{thm}
\begin{proof}
The proof is similar to Theorem $\ref{iso}.$
Since $\phi$ restricts to an injective map on $V$ and $R \hookrightarrow A$, we have $ker(\phi) \cap R = \{0\}.$ 
Then it follows from \cite[Theorem 2.7]{CV2} that $\phi$ is injective.

Note that $\rank[\End(O)]  = 64$ because $\rank(O) =8.$ Since $M_{2^{n+1}}(\End(O)) = M_{2^{n+1}}(R) \otimes \End(O)$, one can see that its rank is $2^{2n+8}$ which is the same as $\rank(Cl) = 2^{\rank(V)}.$ 
\end{proof}

The algebra $M_{2^{n+1}}(\End(O))$ has a $\mathbb{Z}_2$-grading where the the even and odd elements are matrices of the form
$(\begin{smallmatrix}
*    &  0  \\
0       &*  \end{smallmatrix})$ and $(\begin{smallmatrix}
 0   &  *  \\
*        & 0 \end{smallmatrix})$. The homomorphism $\phi$ is therefore a $\mathbb{Z}_2$-graded homomorphism.
Here too, if we restrict ourselves to non-singular quadratic forms, then the map $\phi$ (being a graded-homomorphism between graded Azumaya algebras of the same rank) is an isomorphism \cite[Lemma 6.7.10]{HM}. The paper \cite{CV2} analyzes such \emph{embeddings} for general quadratic spaces, in particular describing the structure of the Clifford algebra and Spin groups.

%

%

\subsection{Composition in $H(R^5)$ using Octonion multiplication.}
Let $v = (a, v_1) $ and $w = (b,w_1)$, where $(v_1, w_1) \in H(R^4).$ Let us identify elements of $H(R^4)$ with the elements of the split octonion algebra -  write $O_i = (v_i, w_i) $ with $q(O_i) = O_i\overline{O_i} = v_i \cdot w_i^{\intercal }.$ 
Let $X = \begin{psmallmatrix}
a     &  O_1 \\  
-\overline{O_1}       &  b\end{psmallmatrix}$ and 
$Y = \begin{psmallmatrix}
a   &  O_2\\   
-\overline{O_2}       &  b'\end{psmallmatrix}.$
\vskip 1mm

\noindent When $q_X = q_Y =1$, we have
\[X \odot Y
=
\begin{psmallmatrix}
a              &  O_1O_2\\\\    
-\overline{O_1O_2}       &    \hskip 5mm  b+ b' -abb'\end{psmallmatrix}\]
The product $X \odot Y$ corresponds to another pair $(v',w') \in H(R^5)$ and  $q(X \odot Y) =  v'\cdot w'^{\intercal } = 1.$ This composition is obviously non-associative.

\vskip 5mm

\textbf{Acknowledgements.} The author is grateful to the referee whose inputs improved the quality of the paper (both stylistically and mathematically).  The author would also like to thank Ravi Rao, Jean Fasel and Anand Sawant for all the conversations and constant encouragement. 

Special thanks to the Euler International Mathematical Institute, Russia for their generous invitation to the Algebraic groups conference in 2019, and for their excellent hospitality in St. Petersburg -  where some of the early thoughts in the paper took a clear shape.  
The author was supported by the DST-Inspire Fellowship in India.


\begin{thebibliography}{FKRS04}
\vskip 2mm 

\bibitem[AHS]{AHS} A. Ambily, R. Hazrat, B. Sury (eds) Leavitt Path Algebras and Classical K-Theory. Indian Statistical Institute Series. Springer, Singapore, 2020. 

\bibitem[AF1]{AF1} A. Asok and J. Fasel, Algebraic vector bundles on spheres, J. Topol., 7(3) :894--926, 2014.

\bibitem[AF2]{AF2} A. Asok and J. Fasel, An explicit KO-degree map and applications, J. Topology 10 (2017), 268--300

\bibitem[AF3]{AF3} A. Asok and J. Fasel, Euler class groups and motivic stable cohomotopy (with an appendix by M. K. Das), To appear J. Eur. Math. Soc.

\bibitem[Ba]{Ba} J. C. Baez, The octonions, Bull. Amer. Math. Soc. (N.S.) 39 (2002), no. 2, 145--205.

\bibitem[B1]{B1} H. Bass. Lectures on topics in algebraic K-theory, Notes by Amit Roy, 
 Tata Institute of Fundamental Research Lectures on Mathematics, 41, Tata Institute of Fundamental Research, Bombay, 1967.
 
 \bibitem[B2]{B2} H. Bass, Clifford algebras and Spinor norms over a commutative ring. Amer. J. Math. 96:156--206, 1974.

\bibitem[Bh1]{Bh1} Manjul Bhargava,  Higher composition laws I: A new view on Gauss composition, and quadratic generalizations. Ann. Math. 159, 217--250 (2004)

\bibitem[BRS]{BRS} S. M. Bhatwadekar and Raja Sridharan, Euler class group of a Noetherian ring, Compositio Math.
122 (2000), 183--222.

\bibitem[CR1]{CR1} P. Chattopadhyay, R.A. Rao, Equality of linear and symplectic orbits, JPAA 219, (2015), 5363--5386. 


\bibitem[CR2]{CR2} P. Chattopadhyay, R.A. Rao, Elementary symplectic orbits and improved $K_1$-stability, Journal of K-Theory 7, (2011), 389--403. 


\bibitem[CV1]{CV1} Vineeth Chintala, On Suslin Matrices and Their Connection to Spin Groups,
Documenta Math. 20 (2015) 531--550

\bibitem[CV2]{CV2} Vineeth Chintala, Embeddings of Quadratic Spaces, Documenta Math. 23 (2018) 1621--1634 

\bibitem[CS]{CS} J. H. Conway, D. A. Smith, On Quaternions and Octonions: Their Geometry, Arithmetic, and Symmetry, A. K. Peters, Ltd., Natick, MA, 2003.

\bibitem[DTZ]{DTZ} Mrinal K. Das, Soumi Tikader and Md. Ali Zinna, Orbit spaces of unimodular rows over smooth
real affine algebras, Invent. Math. 212 (2018), 133--159

\bibitem[F1]{F1} Jean Fasel, Some remarks on orbit sets of unimodular rows, Comment. Math. Helv. 86 (2011), 13--39.

\bibitem[F2]{F2} Jean Fasel, The Vaserstein symbol on real smooth affine threefolds, K-theory, TIFR Publications 19 (2018), 213--224.

\bibitem[FRS]{FRS} J. Fasel, , R. A. Rao., R. G. Swan, On Stably Free Modules over Affine Algebras, Publ.math.IHES (2012) 116: 223--243


\bibitem[G]{G} P. Gille, Octonion algebras over rings are not determined by their norm, Canad. Math. Bull. 57:2 (2014), 303--309.

\bibitem[Gu]{Gu} J. Gubeladze,  Unimodular rows over monoid rings. Adv. Math. 337, 193--215 (2018)

\bibitem[GGR]{GGR} A. Gupta, A. Garge and R. Rao, A nice group structure on the orbit space of unimodular rows II. J. Algebra 407, 201–223 (2014)

\bibitem[GRK]{GRK} N. Gupta , D. R. Rao, S. Kolte (2020) A Survey on the Non-injectivity of the Vaserstein Symbol in Dimension Three, Leavitt Path Algebras and Classical K-Theory. 193-202, Indian Statistical Institute Series. Springer, Singapore. 

\bibitem[HM]{HM} J. Helmstetter, A. Micali, Quadratic Mappings and Clifford Algebras, Birkhäuser Verlag AG, 2008.

\bibitem[HOM]{HOM}  A. J. Hahn, O. T. O'Meara, The Classical Groups and K-Theory, Springer-Verlag, Berlin (1989).

\bibitem[JR1]{JR1} Selby Jose, Ravi A. Rao. A Structure theorem for the Elementary Unimodular Vector group, 
Trans. Amer. Math. Soc. 358 (2005), no.7, 3097–3112. 

\bibitem[JR2]{JR2} Selby Jose, Ravi A. Rao. A fundamental property of Suslin matrices, 
Journal of K -theory: K -theory and its Applications to Algebra, Geometry, and Topology 5 (2010), no. 3, 407–436. 

\bibitem[K]{K} Martin Kneser, Composition of binary quadratic forms. J. Number Theory 15 (1982),
no. 3, 406–413.

\bibitem[Kn]{Kn} M.-A. Knus (1991). Quadratic and Hermitian Forms over Rings. Grundlehren der Mathematischen Wissenschaften. Vol. 294. Berlin: Springer-Verlag.

\bibitem[KOS]{KOS} M.-A. Knus, M. Ojanguren, R. Sridharan, Quadratic forms and Azumaya algebras, J. reine angew. math. 303/304 (1978), 231--248.


\bibitem[Ko]{Ko} Vijay Kodiyalam, On the genesis of a determinental identity, Ramanujan Math. Soc. 28 (2013), no. 2, 173--178.

\bibitem[KM]{KM} N. M. Kumar and M. P. Murthy, Remarks on unimodular rows, In `Quadratic forms, linear algebraic groups and cohomology', Developments in Mathematics 18, Springer (2010) 287--293.


\bibitem[Lam]{Lam} T. Y. Lam. Serre's problem on projective modules, Springer Monographs in Mathematics, Springer-Verlag, Berlin, 2006.

\bibitem[L]{L} S. Lang, Algebra, revised third edition, Graduate Texts in Mathematics, vol. 211, Springer-Verlag, New York, 2002, xvi+914 pp.

\bibitem[MD]{MD} B. R. McDonald, Linear Algebra over Commutative Rings, Marcel Dekker, Inc., New York, Basel, 1984.

\bibitem[RBJ]{RBJ}R. A. Rao, R. Basu, S. Jose, Injective Stability for K1 of the Orthogonal group, J. Algebra 323 (2010), 393–-396

\bibitem[RJ]{RJ} Ravi Rao, Selby Jose (2016) A Study of Suslin Matrices: Their Properties and Uses. 
In: Algebra and its Applications. Springer Proceedings in Mathematics \& Statistics, vol 174. Springer, Singapore

\bibitem[S]{S} A. A. Suslin, Stably Free Modules. 
(Russian) Math. USSR Sbornik 102 (144) (1977), no. 4, 537–550. Mat. Inst. Steklov. (LOMI) 114 (1982), 187–195.

\bibitem[S2]{S2} A. A. Suslin, On the structure of the special linear group over polynomial rings, Math. USSR Izv. 11 (1977), 221--238.

\bibitem[SS]{SS} S. Sharma, On completion of unimodular rows over polynomial extension of finitely generated rings over $\mathbb{Z}$, Journal of Pure and Applied Algebra, vol. 225 (2), 2021.

\bibitem[SV]{SV}  A. A. Suslin, L. N. Vaserstein, Serre's problem on projective modules over polynomial rings, and algebraic K-theory, Izv. Akad. Nauk. SSSR Ser. Mat. 40 (1976), 993--1054

\bibitem[SVe]{SVe} T.A. Springer, F.D. Veldkamp, Octonions algebras, Jordan algebras and excep- tional groups, Springer Monographs in Mathematics (2000 ).

\bibitem[TS1]{TS1} Tariq Syed, A generalized Vaserstein symbol, Annals of K-Theory 4 (2019), no. 4, 671--706

\bibitem[TS2]{TS2} Tariq Syed, Symplectic orbits of unimodular rows, arXiv:2010.06669.

\bibitem[vdk1]{vdk1} W. van der Kallen, A group structure on certain orbit sets of unimodular rows, J. of Algebra,
82(1983), 363--397.

\bibitem[vdk2]{vdk2} W. van der Kallen, A module structure on certain orbit sets of unimodular rows, JPAA. 57 (1989), 281--316.

\bibitem[vdk3]{vdk3}W. van der Kallen, From Mennicke symbols to Euler class groups. In Algebra, arithmetic and geometry, Part I, II (Mumbai, 2000), Tata Inst. Fund. Res. Stud. Math. 16, TIFR, Bombay, 2002, 341--354.


\bibitem[Vo]{Vo} John Voight, Quaternion algebras, Springer Graduate Texts in Mathematics series, 2020.

\bibitem[W]{W} Melanie Matchett Wood, Gauss composition over an arbitrary base, Adv. Math. 226 (2) (2011) 1756--1771.

\end{thebibliography}
\end{document}